\documentclass[11pt, reqno]{amsart}
\usepackage{lmodern}
\usepackage{amsmath, amsthm, amssymb, amsfonts}
\usepackage[normalem]{ulem}
\usepackage{hyperref}
\usepackage[all,cmtip]{xy}
\usepackage{verbatim}
\usepackage{longtable}
\usepackage[center]{caption}
\usepackage{mathtools}

\usepackage{tikz-cd}

\theoremstyle{plain}
\newtheorem{thm}{Theorem}[section]
\newtheorem{cor}[thm]{Corollary}

\newtheorem{lemma}[thm]{Lemma}
\newtheorem{prop}[thm]{Proposition}
\newtheorem{conj}[thm]{Conjecture}

\theoremstyle{definition}

\theoremstyle{remark}
\newtheorem{rmk}[thm]{Remark}

\newcommand{\BC}{{\mathbb{C}}}
\newcommand{\BD}{{\mathbb{D}}}

\newcommand{\BN}{{\mathbb{N}}}

\newcommand{\BP}{{\mathbb{P}}}
\newcommand{\BQ}{{\mathbb{Q}}}

\newcommand{\BZ}{{\mathbb{Z}}}

\newcommand{\CC}{{\mathcal C}}

\newcommand{\CE}{{\mathcal E}}
\newcommand{\CF}{{\mathcal F}}

\newcommand{\CK}{{\mathcal K}}
\newcommand{\CL}{{\mathcal L}}
\newcommand{\CM}{{\mathcal M}}

\newcommand{\CO}{{\mathcal O}}

\newcommand{\CX}{{\mathcal X}}

\newcommand{\blangle}{\big\langle}
\newcommand{\brangle}{\big\rangle}

\DeclareMathOperator{\Hilb}{Hilb}

\DeclareFontFamily{OT1}{rsfs}{}
\DeclareFontShape{OT1}{rsfs}{n}{it}{<-> rsfs10}{}
\DeclareMathAlphabet{\curly}{OT1}{rsfs}{n}{it}
\renewcommand\hom{\curly H\!om}
\newcommand\ext{\curly Ext}

\newcommand{\p}{\mathbb{P}}

\newcommand\Proj{\operatorname{Proj}}

\newcommand{\Mbar}{{\overline M}}

\newcommand{\Chow}{\mathrm{CH}}

\newcommand{\Sym}{{\mathrm{Sym}}}

\newcommand{\ev}{{\mathrm{ev}}}

\newcommand{\RN}[1]{%
  \textup{\uppercase\expandafter{\romannumeral#1}}%
}

\begin{document}
\title[Elliptic curves in hyper-K\"ahler varieties]
{Elliptic curves in hyper-K\"ahler varieties}
\date{\today}

\author{Denis Nesterov}
\address{University of Bonn, Institut f\"ur Mathematik}
\email{nesterov@math.uni-bonn.de}

\author{Georg Oberdieck}
\address{University of Bonn, Institut f\"ur Mathematik}
\email{georgo@math.uni-bonn.de}

\begin{abstract}
We show that the moduli space of elliptic curves of minimal degree in a general Fano variety of lines of a cubic fourfold is a
non-singular curve of genus $631$. The curve admits a natural involution with connected quotient.
We find that the general Fano contains precisely $3780$ elliptic curves of minimal degree with fixed (general) $j$-invariant.

More generally, we express (modulo a transversality result) the enumerative count of elliptic curves
of minimal degree
in hyper-K\"ahler varieties
with fixed $j$-invariant in terms of Gromov--Witten invariants.
In $K3^{[2]}$-type this leads to explicit formulas of these counts in terms of modular forms.

\end{abstract}
\baselineskip=14.5pt
\maketitle


\section{Introduction}
\subsection{Moduli of elliptic curves}
A non-singular complex projective variety $X$ is \emph{hyper-K\"ahler}
if it is simply-connected and $H^0(X, \Omega^2_X)$ is spanned by a non-degenerate holomorphic $2$-form.
Let $(X,H)$ be a very general 
polarized hyper-K\"ahler variety.
Since the Picard group of $X$ is of rank $1$, there exist a unique curve class of minimal degree,
\[ H_2(X,\BZ)_{\mathrm{alg}} = \BZ \beta. \]

Let $M_{g,n}(X,\beta)$ be the moduli space of stable maps $f: C \to X$ from non-singular curves of genus $g$ with $n$ distinct markings representing the class $\beta$.
Let $\Mbar_{g,n}(X,\beta)$ be its Deligne-Mumford compactification parametrizing stable maps from nodal curves.
The expected dimension of both of these moduli spaces is
\[ \mathsf{vd} := (\dim X-3)(1-g) + n + 1. \]

In genus one the expected dimension of $\Mbar_{1,0}(X,\beta)$ is $1$. It
parametrizes two types of maps $f: C \to X$:
Either $C$ is irreducible of arithmetic genus $1$,
or $C$ has an elliptic tail which is contracted by $f$.
We expect the following non-degeneracy result for maps from non-singular curves:
\begin{conj} \label{Conj1} Let $X$ be a very general polarized hyper-K\"ahler variety with primitive curve class $\beta$. Then
$M_{1,0}(X,\beta)$ is pure of dimension $1$.
\end{conj}

In genus $0$ the moduli space of stable maps $\Mbar_{0,0}(X,\beta)$ for $X$ very general is pure
of the expected dimension, see for example \cite[Prop.2.1]{OSY}.
For a K3 surface $X$ the moduli space $M_{1,0}(X,\beta)$ is always non-empty and smooth of dimension $1$,
see Section~\ref{Subsection_Calculations_K3}. Hence Conjecture~\ref{Conj1} holds for K3 surfaces.
%

In Conjecture~\ref{Conj1} the moduli space $M_{1,0}(X,\beta)$ is allowed to be empty.
This case occurs for example on deformations of
generalized Kummer fourfolds of a principally polarized abelian surface, see Section~\ref{Kummer}.


\subsection{Counting elliptic curves}
An \emph{elliptic curve} in $X$ is an irreducible curve $C \subset X$ of geometric genus $1$.
We want to count elliptic curves in class $\beta$ with normalization having a fixed $j$-invariant. 
Since we expect the family of elliptic curves in $X$ to be one-dimensional and 
fixing the $j$-invariant is a codimension $1$ condition we expect a finite count.

Concretely, let $n_{X,j_0} \in \BN_0 \cup \{ \infty \}$ be the number of
elliptic curves in class $\beta$ with $j$-invariant of the normalization equal to $j_0 \in M_{1,1}$,
\[ n_{X,j_0} = \# \{ C \subset X\, |\, [C] = \beta,\, j(\tilde{C}) \cong j_0 \}. \]
If the set on the right hand side is infinite, we set $n_{X,j_0} = \infty$.

Below we express, up to a non-degeneracy assumption, the counts $n_{X,j}$ for general $j$ in terms of the Gromov--Witten invariants of $X$.
In several cases these are known and lead to explicit formulas for $n_{X,j}$. 

\subsection{Gromov--Witten theory}
The moduli space of stable maps $\Mbar_{g,n}(X,\beta)$ carries a (reduced) virtual fundamental class
\[ [ \Mbar_{g,n}(X,\beta) ]^{\text{vir}} \in \Chow_{\mathsf{vd}}( \Mbar_{g,n}(X,\beta) ), \]
see \cite{MPT, HilbK3, K3xE}. 
Gromov--Witten invariants of $X$ are defined by integrating the virtual class against
classes pulled back along the natural maps
\[ \ev_i : \Mbar_{g,n}(X,\beta) \to X, i=1, \ldots, n, \quad p : \Mbar_{g,n}(X,\beta) \to \Mbar_{g,n} \]
which evaluate a stable map at the $i$-th marking and forget the map respectively.
We will need two particular Gromov--Witten invariants. 
The first is the virtual analog of the count $n_{X,j}$. Let
\[ N_X = \int_{[ \Mbar_{1,1}(X,\beta) ]^{\text{vir}}} \ev_1^{\ast}(D) \cup p^{\ast} [(E,0)] \]
where $(E,0) \in \Mbar_{1,1}$ is a point and $D \in H^2(X,\BQ)$ is an arbitrary divisor class with intersection pairing $\langle \beta, D \rangle = 1$.
The second invariant is
\[ 
C_X
=
\int_{[ \Mbar_{0,1}(X), \beta) ]^{\text{vir}}} \frac{\ev_1^{\ast} (c( T_X ))}{1 - \psi_1}.
\]
where $c(T_X)$ is the total Chern class of $X$ and $\psi_1$ is the cotangent line class at the first marking.
The denominator is expanded formally:
\[ \frac{1}{1-\psi_1} = 1 + \psi_1 + \psi_1^2 + \ldots .\]

The following result relates the enumerative and virtual counts.
\begin{prop} \label{MainProp} Let $X$ be a hyper-K\"ahler variety and let $\beta \in H_2(X,\BZ)$ be an irreducible curve class.
If every irreducible component of $M_{1,0}(X,\beta)$ is generically reduced of dimension $1$ and every rational curve on $X$ in class $\beta$ is nodal, then
\begin{equation} n_{X,j} = \frac{1}{2} ( N_X - C_X ) \label{formula1} \end{equation}
for any general $j \in M_{1,1}$.
\end{prop}

Here we say that a class $\beta \in H_2(X,\BZ)$ is a curve class if there exists an algebraic curve $C \subset X$ with $[C] = \beta$. The class is \emph{irreducible} if it can not be written as a sum of two non-trivial curve classes.

The factor $1/2$ on the right hand side of \eqref{formula1} arises since we do not identify conjugate maps in the Gromov-Witten integral. If $M_{1,0}(X,\beta)$ is one-dimensional but not necessarily generically reduced,
then the right side of \eqref{formula1} computes the length of the fiber of the forgetful map $M_{1,1}(X,\beta) \to M_{1,1}$ over the point $(E,0)$;
hence we count the elliptic curves with multiplicities.

If $X$ is a K3 surface or a fourfold of $K3^{[2]}$ type\footnote{
A hyper-K\"ahler is of $K3^{[n]}$-type if it is deformation equivalent to the Hilbert scheme of $n$ points on a K3 surface.}
then the Gromov--Witten theory of $X$ is known in all primitive curve classes \cite{MPT,HilbK3}.
For these cases we compute the right hand side of Proposition~\ref{MainProp}
in terms of modular forms in Section~\ref{Section_GW}.
By deformation invariance the Gromov--Witten invariants only depend on the 
Beauville-Bogomolov norm\footnote{The pairing
is induced from the Beauville-Bogomolov form on $H^2(X,\BZ)$ via Poincar\'e duality.
The pairing is $\BQ$-valued in general.} of the class $\beta$, denoted $(\beta,\beta)$.
The first expected values of $n_{X,j}$ are listed in Tables \ref{Table1} and \ref{Table2}.

\begin{table}[t]
{\renewcommand{\arraystretch}{1.5}\begin{tabular}{| c | c | c | c | c | c | c | c | c | c | c |}
\hline
\!$(\beta,\beta)$\! & $<0$ & $0$ & $2$ & $4$ & $6$ & $8$ & $10$ & $12$ \\
\hline
$n_X$ & $0$ & $24$ & $648$ & $9600$ & $102600$ & $881280$ & $6442320$ & $41513472$ \\
\hline
\end{tabular}}
\caption{First values of $n_{X,j}$ for a K3 surface $X$ and general $j$.}
\label{Table1}

{\renewcommand{\arraystretch}{1.5}\begin{tabular}{| c | c | c | c | c | c | c | c | c | c | c |}
\hline
\!$(\beta,\beta)$\! & $<0$ & $0$ & $\frac{3}{2}$ & $2$ & $\frac{7}{2}$ & $4$ & $\frac{11}{2}$ & $6$ \\
\hline
$n_X$ & $0$ & $648$ & $3780$ & $23760$ & $129600$ & $470880$ & $2396520$ & $6629040$ \\ 
\hline
\end{tabular}}
\caption{First values of $n_{X,j}$ in $K3^{[2]}$-type for general $j$, assuming $M_{1,0}(X,\beta)$ is generically reduced of dimension $1$.}
\label{Table2}
\end{table}

We discuss the first cases. The number $24$ in Table~\ref{Table1} is the number of nodal fibers in a general elliptic K3 surface.
The degree zero case in $K3^{[2]}$-type, $648$, is the 
virtual number of elliptic curves of fixed complex structure
in a general Lagrangian fibration $\pi : X \to \p^2$
(the count is virtual here since the moduli space is not of expected dimension; moreover, we work in the setting of complex hyperk\"ahler manifolds that are not necessarily projective).
This can be seen as follows. By a result of Markman the Lagrangian fibration $\pi$ is (a twist of)
the relative Jacobian fibration 
of a genus two K3 surface $(S,L)$.
Under $\pi$ the elliptic curves in $X$ of primitive class map to points on the discriminant\footnote{
If $f:E \to J(C)$ is a map to the Jacobian of a smooth curve, then the image intersects the theta divisor with multiplicity $\geq 2$.
Hence if $C \in |L|$ and so $J(C) \subset X$, then the class of $f(E)$ is not primitive in $X$.
%
},
and the $j$-invariant of such elliptic curve is precisely the $j$-invariant of the corresponding nodal elliptic curve in the genus 2 K3 surface determined by the basepoint.
This explains the equality
\[ n_{K3[2],(\beta,\beta)=0} \ = \ n_{K3,(\beta,\beta)=2}. \]
The number $23760$ 
concerns the double covers of EPW sextics.
The case $(\beta,\beta) = \frac{3}{2}$ corresponds to the Fano variety of lines of a cubic fourfold that we will consider below.

If $X$ is of $K3^{[n]}$-type explicit conjectural formulas for the right hand side of Proposition~\ref{MainProp} can be
obtained from conjectures made in~\cite{HilbK3}.



\subsection{The Fano variety of lines} \label{Subsection_intro_fano_var}
Let $Y \subset \BP^5$ be a nonsingular cubic fourfold and let
\[ F = \{ \, l \in \mathrm{Gr}(2,6) : l \subset Y \} \]
be the Fano variety of lines in~$Y$.
The Fano varieties $F$ form 
a $20$-dimensional family of polarized hyper-K\"ahler fourfolds
of $K3^{[2]}$-type \cite{BD}.
Let $H_{F}$ be the Pl\"ucker hyperplane section on $F$ and consider the primitive integral class
\begin{equation} \beta = \frac{1}{36} H_F^3. \label{3fsd} \end{equation}
The class is of degree $\langle \beta, H \rangle = \frac{1}{36} \deg_{H_F}(X) = 3$ with respect to $H_F$.
If $F$ is very general, then $\beta$ is the unique primitive curve class. 

Consider the projective embedding
\[ F \subset \mathrm{Gr}(2,6) \hookrightarrow \p^{14} \]
defined by $H_F$ and assume that $\beta$ is irreducible (this is a general condition).
If $C \subset F$ is a curve of class $\beta$ then it is of degree $3$ in $\p^{14}$.
It spans a three-plane in $\p^{14}$ if and only if it is a rational normal curve.
Otherwise the curve spans a $\p^2$ in which case it is a plane cubic and hence of arithmetic genus~$1$.
Moreover, since $\p^2 \cap \mathrm{Gr}(2,6)$ contains a cubic curve we see that in fact $\p^2 \subset \mathrm{Gr}(2,6)$.\footnote{
We thank C.~Voisin for pointing out this approach to elliptic curves in $F$.}
Every $\p^2 \subset \mathrm{Gr}(2,6)$ corresponds to lines passing through a fixed point $v \in \p^5$.
Hence the curves in class $\beta$ of arithmetic genus $1$ correspond to surfaces in $Y$ which are cones.

Let $\Chow_{\beta}(F)$ be the Chow variety of curves in $F$ in class $\beta$.
Define
\[ \Sigma \subset \Chow_{\beta}(F) \]
to be the (reduced) locus of curves of arithmetic genus $1$.

\begin{thm} \label{MainThm} Let $F$ be a general Fano variety of lines of a cubic fourfold.
Then $\Sigma$ is a non-singular curve with the following properties:
\begin{enumerate}
\item[a)] $\Sigma$ has at most two connected components.
\item[b)] $\Sigma$ admits a natural involution with connected quotient.
\item[c)] $g(\Sigma) = 631$.
\item[d)] If the curve $C \in \Sigma$ is singular, then $C \subset F$ is a rational nodal curve.
\item[e)] The map $\Sigma \to \p^1$ taking a point $[C] \in \Sigma$ to its $j$-invariant (in the coarse quotient of $\Mbar_{1,1}$) is of degree $3780$.
\end{enumerate}
\end{thm}

We do not know whether the curve $\Sigma$ in the theorem is connected, although we expect it. 
The genus of a (possibly disconnected) curve is defined to be $g(\Sigma) = 1 - \chi(\CO_{\Sigma})$.


Let $F$ be a general Fano variety of lines.
Let 
\[ S \subset \Chow_{\beta}(F) \]
be the locus of rational curves.
Then $S$ is a smooth connected surface isomorphic to the locus of lines in $Y$ of second type \cite{A, OSY}.
We then have
\[ \Chow_{\beta}(F) = S \cup \Sigma \]
and the intersection $S \cap \Sigma$ consists of finitely many (at most $3780$) points corresponding precisely to the nodal rational curves.

Finally we consider the moduli space of stable maps.

\begin{prop} \label{Prop_Non_sing}
Let $F$ be a general Fano variety of lines of a cubic fourfold.
Then $M_{1,0}(F,\beta)$ is a non-singular curve and isomorphic to the open subset of $\Sigma$ parametrizing smooth elliptic curves.
Moreover, every curve in $F$ of class $\beta$ is nodal.
\end{prop}
%

By using Theorem~\ref{MainThm} directly, or combining Propositions \ref{MainProp} and~\ref{Prop_Non_sing},
we have thus obtained two proofs of the following:

\begin{cor}
A general Fano variety of lines contains precisely $3780$ elliptic curves of minimal degree and of fixed general $j$-invariant. 
\end{cor}

\subsection{Plan of the paper}
In Section~\ref{Section_GW} we prove Proposition~\ref{MainProp}
and recall the Gromov--Witten calculations for the K3 and $K3^{[2]}$ case.
In Section~\ref{Section_Fano} we discuss the geometry of the Fano and prove Theorem~\ref{MainThm} and Proposition~\ref{Prop_Non_sing}.
In Section~\ref{Kummer} we give an example where $M_{1,0}(X,\beta)$ is empty.

\subsection{Conventions}
A statement holds for a \emph{general} (resp. \emph{very general}) polarized projective variety $(X,H)$ if it holds away from a Zariski closed subset (resp. a countable union of proper Zariski-closed subsets)
in the corresponding component of the moduli space.

\subsection{Acknowledgements}
We would like to thank C.~Voisin for suggesting the approach to calculate the genus of the curve $\Sigma$,
and H.-Y. Lin who pointed out the generalized Kummer example.
We further thank L.~Battistella, D.~Huybrechts, R.~Mboro, J.~Schmitt, J.~Shen, and Q.~Yin for useful discussions and their interest.
We are also very grateful to the anonymous referees for helpful comments.

\section{Gromov--Witten theory}
\subsection{Proof of Proposition~\ref{MainProp}} \label{Section_GW}

Let $(E,0)$ be a general non-singular elliptic curve and let
$\Mbar_{E,n}(X,\beta)$ be the moduli space of $n$-pointed
stable maps from $E$ to $X$ in class $\beta$.
The points of $\Mbar_{E,n}(X,\beta)$ correspond
to stable maps from nodal, $n$-pointed degnerations of $E$ to $X$,
see \cite{Pand} for a definition\footnote{In our case we do not identify maps that differ by a conjugation on $E$}.
The moduli space $\Mbar_{E,n}(X,\beta)$ carries a perfect obstruction theory
and its reduced virtual class \cite{MPT, HilbK3, K3xE} satisfies
\[ [ \Mbar_{E,n}(X,\beta) ]^{\text{vir}} = p^{\ast}[(E,0)] \cap [ \Mbar_{1,n}(X,\beta) ]^{\text{vir}}. \]
Using this and the divisor equation \cite{FP} we obtain
\[ N_X 
=
 \int_{[ \Mbar_{E,1}(X,\beta) ]^{\text{vir}}} \ev_1^{\ast}(D)
=
\int_{[ \Mbar_{E,0}(X,\beta) ]^{\text{vir}}} 1.
\]

Since $\beta$ is irreducible the moduli space $\Mbar_{E,0}(X,\beta)$ parametrizes the following two types of stable maps $f : C \to X$:
\begin{enumerate}
\item $C$ is non-singular, isomorphic to $E$, and $f : C \to X$ is birational onto an elliptic curve in $X$, or
\item $C = \p^1 \cup_x E$ and the degree of $f|_{E}$ is zero.
\end{enumerate}

Since by assumption every irreducible component of $M_{1,0}(X,\beta)$ is generically reduced of dimension $1$,
there exists a dense open subset 
\[ U \subset M_{1,0}(X,\beta) \]
which is smooth.
Consider the map $\psi : U \to \p^1$ which sends a stable map to the $j$-invariant of its source.\footnote{Here
$\p^1$ is the coarse moduli space of the Deligne-Mumford stack $\Mbar_{1,1}$.
After an \'etale cover the universal curve over $U$ admits a section and defines a map to $\Mbar_{1,1}$, and hence by composition to $\p^1$.
Since $\p^1$ is a scheme these morphisms glue to an actual map $U \to \p^1$.}
By generic smoothness of $\psi$, the curve $E$ is general, since the complement of $U$ in $M_{1,0}(X,\beta)$ consists of finitely many points,
it follows that the locus $M_1$ in $\Mbar_{E,0}(X,\beta)$ parametrizing maps of the first kind consists of finitely many reduced points.
By definition $M_1$ is a closed subset of $\Mbar_{E,0}(X,\beta)$.
On the other hand, since the locus of maps from reducible curves is a closed subset of $\Mbar_{1,0}(X,\beta)$,
we find that $M_1$ is also open in $\Mbar_{E,0}(X,\beta)$. Hence we have the disjoint union:
\[ \Mbar_{E,0}(X,\beta) = M_1 \sqcup M_2 \]
where the closed points of $M_2$ parametrize maps of second type.
Splitting the contribution from the virtual class hence yields
\[ N_X = \int_{[M_1]^{\text{vir}}} 1 + \int_{[M_2]^{\text{vir}}} 1. \]

We now determine the contributions from each component.
We consider $M_1$ first which consists of finitely many reduced points.
Since the dimension of the virtual fundamental class equals the dimension of the moduli space (zero), it coincides with the usual fundamental class and thus its integral is the number of these reduced points.
Since conjugate maps yield the same elliptic curve in $X$ we conclude
\[ n_{X, j(E)} = \frac{1}{2} |M_1| = \frac{1}{2} \int_{[M_1]^{\text{vir}}} 1. \]

For the contribution from the second component consider the morphism
\[ \mathfrak{s} : M_{0,1}(X,\beta) \to M_2 \]
that sends a stable map $(g : \p^1 \to X, x)$ to the element $(f : \p^1 \cup_x E \to X)$ with $f|_{\p^1} = g$ and $f|_{E}$ constant.
By construction, $\mathfrak{s}$ is an isomorphism on closed points.
We claim that $\mathfrak{s}$ is also scheme-theoretically an isomorphism.
For this it is enough to show that all infinitesimal smoothings of the domain $C = \p^1 \cup_x E$ of $f$ are obstructed in $M_2$.
Recall the deformation obstruction sequence for $p = [f:C \to X]$ in $\Mbar_{1,0}(X,\beta)$:
\[ H^0(C, f^{\ast} T_X) \to T_{\Mbar_{1,0}(X,\beta), p} \to H^1(C, T_C) \to H^1(C, f^{\ast} T_X) \to \mathrm{Obs}_{p} \to 0 \]
where $\mathrm{Obs}_p$ is the fiber of the obstruction sheaf at the point $p$. We refer to \cite[Part 4]{MirSym} for more background.
By the normalization sequence $0 \to \CO_C \to \CO_{\p^1} \oplus \CO_{E} \to \BC_{x} \to 0$
for the curve $C$ we have the decomposition
\[ H^1(C, T_C) = H^1(E, T_E) \oplus (T_{x} \p^1 \otimes T_{x} E). \]
The first summand corresponds to deformation of the complex structure of the elliptic curve, the second correspond to smoothing of the nodes.
Since in $\Mbar_{E,0}(X,\beta)$ the former are obstructed, we get the sequence
\[ H^0(C, f^{\ast} T_X) \to T_{M_1, p} \to T_{x} \p^1 \otimes T_{x} E \xrightarrow{\rho} H^1(C, f^{\ast} T_X) \to \mathrm{Obs}_{p} \to 0. \]
To prove that any smoothings of $C$ are obstructed we need to show that $\rho$ is injective.
The composition
\begin{multline*} T_x \p^1 \cong T_{x} \p^1 \otimes T_{x} E \overset{\rho}{\to} H^1(C, f^{\ast}T_X) \\
\overset{\sigma}{\to} H^1(\p^1, f|_{\p^1}^{\ast}T_X) \oplus H^1(E, f|_{E}^{\ast}T_X) \to H^1(E, f|_{E}^{\ast} T_X) \cong T_{X, f(x)} \end{multline*}
is precisely the differential of $f|_{\p^1}$ at $x$, where $\sigma$ is induced by the normalization sequence.
But by our second assumption in Proposition~\ref{MainProp} every rational curve in $X$ is at most nodal. Hence $f|_{\p^1}$, which is the normalization of its image, has everywhere non-zero differential. This proves the claim.\footnote{We refer also to \cite[Section 2]{Pand} for a similar argument: We can smooth out elliptic tails only if the image has non-nodal singularieties, for example a cusp.}

To determine the contribution from $M_2$ to our integral, it remains to relate the reduced virtual classes on both sides of the isomorphism
\[ \mathfrak{s} : M_{0,1}(X,\beta) \overset{\cong}{\to} M_2. \]
The obstruction sheaf $\mathrm{Obs}'$ of $M_{0,1}(X,\beta)$ has fiber $H^1(\p^1, f^{\ast} T_X)$ at $p$. Hence the
obstruction sheaves $\mathrm{Obs}$ and $\mathrm{Obs}'$ differ
 by the cokernel of the inclusion
\[ T_{x} \p^1 \otimes T_{x} E \to H^1(C, f^{\ast} T_X) \to H^1(E, f|_{E}^{\ast} T_X) = H^1(E,\CO_E) \otimes \ev_x^{\ast} T_X. \]
By the excess intersection formula this gives\footnote{In the above discussion we have worked with the
standard (non-reduced) perfect obstruction theory. The reduction process on both $\Mbar_{E,0}(X,\beta)$ and $\Mbar_{0,1}(X,\beta)$
is compatible and does not affect the excess intersection calculations.}
\begin{align*} [ M_2 ]^{\text{vir}} 
& = \left[ \frac{ \ev_1^{\ast} c(T_X) }{1-\psi_1} \right]_{2d-1} \cap [ \Mbar_{0,1}(X,\beta) ]^{\text{vir}} 
\end{align*} 
where $\psi_1$ is the first Chern class of the cotangent line bundle at the first marking, and we hae used that $H^1(E,\CO_E)$ and $T_{x}E$ are constant over the moduli space. 
This yields
\begin{align*}
\int_{[ M_2 ]^{\text{vir}}} 1
& = \int_{[ \Mbar_{0,1}(X,\beta) ]^{\text{vir}}} \left[ \frac{\ev_1^{\ast} c( T_X )}{1 - \psi_1}  \right]_{2d-1} 
= C_X. 
\end{align*}
Putting everything together we find
$N_X = 2 n_{X, j_{\text{gen}}} + C_X$. \qed

\subsection{Calculations: K3 surfaces} \label{Subsection_Calculations_K3}
Let $X$ be a very general K3 surface with primitive curve class $\beta \in H_2(X,\BZ)$ of self-intersection
\[ \langle \beta, \beta \rangle = 2h-2, \quad h \in \BZ_{\geq 0}. \]
We claim that the moduli space $M_{1,0}(X,\beta)$ is a smooth curve.
Indeed, let $p = [f : E \to X]$ be a point on the moduli space and consider the sequence $0 \to T_E \to f^{\ast} T_X \to \CL \to 0$
for some degree $0$ line bundle $\CL$ on the elliptic curve $E$.
It follows that $T_{M_{1,0}(X,\beta),p} = H^0(E, \CL)$ is of dimension $\leq 1$.
Using that the dimension has to be at least the expected dimension (which is $1$), we see that the tangent space is exactly $1$-dimensional which gives the claim.
%
Moreover, by a result of Chen, every rational curve in class $\beta$ is nodal \cite{ChenK3}.
Hence the assumption of Proposition~\ref{MainProp} are satisfied and we can determine the number of elliptic curves via Gromov--Witten theory.

For that recall that by deformation invariance the invariants $N_X$ and $C_X$ only depend on the number $h$.
By Proposition~\ref{MainProp} the same holds for $n_{X,j}$ for general $j$. We write
\[ N_X = N_{K3,h}, \quad C_X = C_{K3,h}, \quad n_{X,j} = n_{K3,h}. \]
By the Yau-Zaslow formula (proven in \cite{Bea, BL}) and a basic Gromov--Witten calculation we have
\begin{align*} 
\sum_{h \geq 0} N_{K3,h} q^{h-1} & = 2 q \frac{d}{dq} \frac{1}{\Delta(q)} \\
\sum_{h \geq 0} C_{K3,h} q^{h-1} & = -2 \frac{1}{\Delta(q)}
\end{align*}
where
$\Delta(q) = q \prod_{m \geq 1} (1-q^m)^{24}$
is the discriminant modular form.
Hence
\[ \sum_{h = 0}^{\infty} n_{K3,h} q^{h-1} = \frac{1}{\Delta(q)} + q \frac{d}{dq} \frac{1}{\Delta(q)}. \]

\subsection{Calculations: $K3^{[2]}$ type}
Let $X$ be a hyper-K\"ahler variety of $K3^{[2]}$ type.
Below we will compute the right hand side of Proposition~\ref{MainProp}.

By deformation invariance (see \cite{OSY}) the Gromov--Witten invariants in a primitive class $\beta$
only depend on its Beauville-Bogomolov-square
\[ (\beta, \beta) = \frac{s}{2}, \quad s \in \BZ,\, s \equiv 0,3 (\text{mod } 4).\footnote{If $X$ is of $K3^{[2]}$-type, then
as a lattice $H_2(X,\BZ) \cong H_2(S,\BZ) \oplus \BZ a$ where $S$ is a K3 surface and $(a,a) = -1/2$. Hence the cases $s=1,2$ mod $4$ does not occur.} \]
We write $N_X = N_{K3[2],s}$ et cetera.
To evaluate the Gromov--Witten invariants we specialize
to the Hilbert scheme of 2 points of an elliptic K3 surface.
The value of $N_{K3[2],s}$ is then given precisely by \cite[Prop.1]{HilbK3}.
Hence it remains to evaluate $C_X = C_{K3[2],s}$. By definition and since odd Chern classes vanish on a hyperk\"ahler variety we have
\[ C_{X} = \langle \tau_1( c_2(X) ) \rangle^X_{0,\beta} + \langle \tau_3(1) \rangle^X_{0,\beta} \]
where we used the standard bracket notation for Gromov--Witten invariants:
\[ \langle \tau_{k_1}(\gamma_1) \ldots \tau_{k_n}(\gamma_n) \rangle^X_{g,\beta} = \int_{ [ \Mbar_{g,n}(X,\beta) ]^{\text{vir}} } \prod_i \ev_i^{\ast}(\gamma_i) \psi_i^{k_i}. \]
Since we are in genus $0$ we can apply topological recursion relations \cite{FP} to remove $\psi$-classes from the integral.
For example,  for any $\alpha \in H^4(X)$ and $D \in H^2(X)$ we have
\[ (D \cdot \beta)^2 \blangle \tau_1(\alpha) \rangle_{0,\beta} = \blangle \alpha, D^2 \rangle_{0,\beta} - 2 (D \cdot \beta) \blangle \tau_0(\alpha \cdot D) \rangle_{0,\beta}. \]
Choosing $D$ such that $D \cdot \beta = 1$ yields in our case
\[ \blangle \tau_1(c_2) \rangle_{0,\beta} = \blangle c_2, D^2 \rangle_{0,\beta} - 2 \blangle \tau_0( c_2 \cdot D ) \brangle_{0,\beta}. \]
The right hand side involves one and two-point invariants with no $\psi$-classes.
These can be evaluated directly from \cite[Thm.10]{HilbK3}, see also \cite[App.~A]{OSY} for a different presentation.
The case $\langle \tau_3(1) \rangle_{0,\beta}$ is similar.

Putting things together we arrive at the evaluation
\begin{multline} \label{ddcdc} 
\sum_{n,k} n_{K3[2],\frac{1}{2} (4n-r^2)} q^n y^k
=
\frac{\Theta^2}{\Delta} \left[ 54 \wp E_2 - \frac{9}{4} E_2^{2} + \frac{3}{4} E_4 - 54 \wp + \frac{15}{2} E_2 - 6 \right] \\
= 
648 + (648 y^2 + 3780 y + 23760 + 3780 y^{-1} + 648 y^{-2}) q + O(q^2).
\end{multline}
where $\Theta$, $\wp$, $E_k$ are the Jacobi theta function, Weierstra{\ss} elliptic function and the Eisenstein series
taken in the convention of \cite[App.B]{HilbK3}.

Finally, we consider the generating series
\[ n_{K3[2]}(q) = \sum_{s} n_{K3[2],s} q^s. \]
Applying the correspondence between
Jacobi forms of index $1$ and modular forms as explained in \cite[Thm.5.4]{EZ},
equality \eqref{ddcdc} becomes
%
\begin{multline*}
n_{K3[2]}(-q) \\
=
\frac{1}{F(q) \vartheta(q) \Delta(q^4)}
\left[
-\frac{9}{4} (\vartheta^4 + 4F)(G_2 - 1) - \frac{3}{8} G_4 + \frac{9}{8} G_2^2 - \frac{15}{4} G_2 + 3
\right],
\end{multline*}
where we have suppressed the argument of $q$ in the bracket on the right, and
\[ \vartheta(q) = \sum_{n \in \BZ} q^{n^2}, \quad 
 F(q) = \sum_{\substack{\text{odd }n > 0 \\ d|n}} d q^n, \quad
G_k(q) = 1 - \frac{2k}{B_k} \sum_{\substack{n \geq 1 \\ d|n}} d^{k-1} q^{4n}
\]
are (quasi-)modular forms for $\Gamma_0(4)$
with $B_k$ the Bernoulli numbers.

\section{The Fano varieties of lines} \label{Section_Fano}
\subsection{Overview} \label{Section_Fano_Overview}
Let $F$ be the Fano variety of lines on a smooth cubic fourfold $Y$.
Consider the incidence correspondence
\[\begin{tikzcd}
\mathcal{L} \arrow{r}{q_{Y}} \arrow[swap]{d}{q_{F}} & Y  \\
F 
\end{tikzcd}
\]
Every curve $C \subset F$ yields via the correspondence a surface in $Y$,
\[ S_{C}:=q_{Y}(q_{F}^{-1}(C)). \]

If the curve $C$ is of class $\beta$, the surface $S_C$ is of degree
$\langle \beta, H_F \rangle = 3$
and thus has cohomology class $H^2_Y \in H^{4}(Y,\mathbb{Z})$,
where $H_Y$ is the hyperplane class on $Y$, see \cite{BD}.
In particular, if $\beta$ is reducible, then the surface $S_C$ is reducible, which in turn implies that $Y$ contains a $\p^2$,
so by \cite{H} we get that $Y$ lies in a divisor in the moduli space of cubic fourfolds.
Since we are interested here in the general case, we may assume that $Y$ does not contain a plane and hence that $\beta$ is irreducible.

Following a discussion of Amerik \cite{A} we give a second description of the elliptic curves in class $\beta$.
For every point $y \in Y$ let
$Z_y \subset Y$ 
be the cone spanned by the lines in $Y$ incident to $y$.
The cone $Z_y$ is the intersection 
of the cubic fourfold $Y=V(f)$,
the tangent plane $T_{y}Y=V(\sum X_{i}\partial_{i}f(y))$
and the polar quadric $R_{y}Y=V(\sum y_{i} \partial_{i}f)$,
\begin{equation} Z_{y}=T_{y}Y\cap R_{y}Y \cap Y. \label{Zy} \end{equation}
The base $B_y$ of the cone is obtained from $Z_y$ by intersecting with a hyperplane $\mathbb{P}^{4}\subset \BP^{5}$ which does not contain $y$.
Since the base is in one-to-one correspondence with the lines through $y$ we have $B_y \subset F$.

If the hypersurfaces in \eqref{Zy} intersect properly then the base is a $(2,3)$-complete intersection curve in
\[ \mathbb{P}^{3}=T_{y}Y\cap \mathbb{P}^{4}. \]
By a result of Amerik \cite{A} this curve is of class $2 \beta$ in $F$.
However, assume that the quadric $T_{y}Y\cap R_{y}Y$ degenerates to a union of two distinct planes.
Since $Y$ does not contain a plane, the base splits as a union
\begin{equation} B_y = E_1 \cup E_2 \label{E12} \end{equation} 
of two distinct planar cubics $E_i$ meeting each other in three points.
Each $E_i$ is an arithmetic genus $1$ curve of class $\beta$ in $F$.

We consider other possible degenerate intersections of \eqref{Zy}.
Since $Y$ does not contain a plane it also does not contain a quadric surface, 
so it always intersects the quadric $R_y Y \cap T_y Y$ properly.
Hence the only other degenerate intersection is when $R_y Y \cap T_y Y$ degenerates further, either to a double plane or to a $3$-plane.
In Lemma~\ref{Lemma32} below we prove that both cases are ruled out for the general Fano.

By the discussion in Section~\ref{Subsection_intro_fano_var} every curve in $F$ of class $\beta$ and arithmetic genus $1$
arises from a cone in $Y$. Hence from the above we conclude that they must be one of the $E_i$ in \eqref{E12}.\footnote{
For an alternative proof one can also use the curve-surface correspondence and classification of cubic surfaces as in \cite[App.B]{OSY}.}

Consider the locus in $Y$ where the cone is degenerate
\begin{equation} \CM_{Y} = \{ y \in Y \, | \, T_{y}Y\cap R_{y}Y \text{ is not integral or of dimension}\geq 3 \}, \label{M_Y_def} \end{equation}
endowed with the reduced subscheme structure.
Recall also the reduced locus $\Sigma \subset \Chow_{\beta}(F)$ of curves of arithmetic genus $1$. 
By the above we have a morphism
\[ \Sigma \to \CM_Y \]
sending an elliptic curve in $F$ to the vertex of the corresponding cone in $Y$.
In Section~\ref{subsection_MY} we study the locus $\CM_Y$ and prove it is non-empty and connected.
In Section~\ref{Subsection_Parametrizing_Cubic_cones} we then express $\Sigma$ directly as the section of a vector bundle in a homogeneous space.
Both descriptions together will yield the description of $\Sigma$ as in Theorem~\ref{MainThm}.
In Section~\ref{Subsection_moduli_of_stable_maps} we consider the deformation theory of the moduli space of genus $1$ stable maps to $F$.

\subsection{} 
\label{subsection_MY}

We first prove connectivity and non-emptiness of $\CM_Y$.
\begin{lemma} \label{Lemma31} Let $Y$ be a smooth cubic fourfold. Then $\CM_{Y}$ is connected and non-empty.
\end{lemma}
\begin{proof} Assume $y=[1: 0 : \ldots  :0]$ and $T_{y}Y=V(X_{1})$. If $Y=V(f)$ for a cubic homogeneous polynomial $f$,
then
\begin{equation} f = X_0^2 X_1 + X_0 f' + f'' \label{dsdf} \end{equation}
where $f', f'' \in \BC[X_{1},...,X_{5}]$
are of degree $2$ and $3$ respectively.
Hence 
$$T_{y}Y \cap R_{y}Y=V(X_{1}) \cap V(f'+2X_{1}X_{0})=V(f'|_{X_{1}=0}).$$
Since this locus is cut out by a quadratic polynomial,
we have $y \in \mathcal{M}_{Y}$ 
if and only if the symmetric matrix
$(f'_{i,j})_{i,j \geq 2}$
has rank $\leq 2$.
Here $f'_{i,j}$ is the coefficient of the monomial $X_{i}X_{j}$ in $f'$.

The Jacobian $(\partial_{i}\partial_{j}f(y))_{i,j}$ of $f$ evaluated at $y$ takes the form 
\[\begin{vmatrix}
0 & 2 & 0  & \dots & 0 \\
2 &  f'_{1,1} & f'_{1,2}  & ...& f'_{1,5} \\
 0  & f'_{2,1} & \ddots & &    \reflectbox{$\ddots$} \\
 \vdots  & \vdots & & (f')_{i,j \geq 2} &  \\
 0  & f'_{5,1} & \reflectbox{$\ddots$} &  &  \ddots
\end{vmatrix} \]
The first two rows of the matrix are linearly independent, so it follows that
\begin{equation} \mathrm{rk}\, (\partial_{i}\partial_{j}f(y))_{i,j} \, = \, 2 + \mathrm{rk}\, (f'_{i,j})_{i,j \geq 2}. \label{rkrk} \end{equation}
Globally, the Jacobian of $f$ is a symmetric form 
$$ J(f):  \CO_{Y}^{6}\otimes \CO_{Y}^{6} \rightarrow \mathcal{O}(1)$$
and hence we have obtained the following global description of $\CM_{Y},$
$$\CM_{Y}=\{ y\in Y | \text{rk}(J(f)_{y})\leq4\}.$$
The expected dimension of this degeneracy locus is 1. By the main result of \cite{Tu} ampleness of $\mathrm{Sym}^{2}(\CO^{6}) \otimes \CO(1)$ implies connectedness of the associated degeneracy locus $\CM_Y$. In turn, by the main result of \cite{G}, ampleness of $\mathrm{Sym}^{2}(\CO^{6}) \otimes \CO(1)$ implies also that $\CM_{Y}$ is non-empty.
\end{proof}

 \begin{lemma} \label{Lemma32}
Let $Y$ be a general cubic fourfold. Then $\CM_{Y}$ is a smooth curve.
Moreover, there exists no $y \in Y$ such that $T_y Y \subset R_{y} Y$ is a double plane or $3$-dimensional. 
 \end{lemma}

\begin{proof} Define the locally-closed subscheme
\[ \Omega =\{ (y,Y) \, | \, Y \text{ smooth cubic fourfold}, y \in \CM_{Y} \} \subset \BP^{5}\times \BP^{55}. \]
The projections from $\Omega$ to both factors yield a correspondence 
 \[\begin{tikzcd}
\Omega \arrow{r}{q} \arrow[swap]{d}{p} & \mathbb{P}^{55}  \\
 \BP^{5}
\end{tikzcd}
\]
such that $p^{-1}(y)=\{Y| y\in \CM_{Y}\}$ and $q^{-1}(Y)=\CM_{Y}$.
A change of coordinates identifies different fibers of $p$. Hence $\Omega$ is a fiber space over $\BP^{5}$. 

Assume as before that $y=[1: 0 : \ldots  :0]$ and $T_{y}Y=V(X_{1})$, so that
$Y$ is defined by a polynomial of the form \eqref{dsdf}.
By the proof of Lemma~\ref{Lemma31} $y \in \CM_{Y}$ if and only if rk$((f'_{i,j})_{i,j \geq 2})\leq 2$.
This condition defines a 12-dimensional locus $\BD_{2} \subset \BC[X_{1},...,X_{5}]_{2}$,
which is smooth at quadrics with rk$((f'_{i,j})_{i,j\geq2})=2$ and
singular along the codimension 3 locus of quadrics with rk$((f'_{i,j})_{i,j\geq 2}) \leq 1$.
The codimension $3$ locus corresponds precisely to the cubics $Y$ such that $T_y Y \cap R_yY$ is a double plane or 3-dimensional.
As there are no conditions on $f''$ or on $f'''$ it follows that
$\{Y| y\in \CM_{Y}\}$ is the (non-empty) open subset of the projective space
$$\p\left( k[X_{1},...,X_{5}]_{3} \times \BD_{2} \times k[X_{1},...,X_{5}]_{1} \right)$$
corresponding to smooth cubic fourfolds.
Adding up dimensions we obtain $\dim p^{-1}(y) = 51$.
Therefore $\Omega$ is 56-dimensional.
Moreover, $\Omega$ is smooth outside the codimension 3 locus of lower rank quadrics.
 
By Lemma~\ref{Lemma31} the locus $\CM_{Y}$ is non-empty for all smooth cubic 4-folds.
In particular $p: \Omega \rightarrow \BP^{55}$ is dominant. Therefore a general
fiber $q^{-1}(Y)=\CM_{Y}$ of $q$ is smooth and 1-dimensional.
The singular locus does not dominate. 
\end{proof}

We record the following consequence.
\begin{cor} \label{Cor_abc}
Let $F$ be the Fano variety of a general cubic fourfold $Y$.
\begin{itemize}
\item[(a)] $\CM_Y$ is a smooth, connected, non-empty curve.
\item[(b)] Every fiber 
of $\Sigma \to \CM_Y$ consists of precisely two points.
\item[(c)] If $\Sigma$ is Cohen-Macaulay, then it is smooth.
\end{itemize}
\end{cor}
\begin{proof}
The first two parts follow directly from Lemma~\ref{Lemma31} and~\ref{Lemma32}.
For the last part, if $\Sigma$ is Cohen-Macaulay, then by miracle flatness we have that $\Sigma \to \CM_Y$
is flat and finite of degree $2$. By part (b) it is also unramified, hence $\Sigma$ is smooth.
\end{proof}

\begin{rmk} 
Consider the Gauss map of a cubic fourfold $Y$,
\[\gamma: Y \rightarrow \mathrm{Gr}(5,6), \ y \mapsto T_y Y \]
The derivative of the map can be expressed as a symmetric form on the tangent bundle $TY$ with values in $\CO(1)$,
$$\RN{2}_{Y}:TY\otimes TY \rightarrow \CO(1).$$
Let $f$ be a defining equation of $Y$ in a standard affine chart on $\BP^{5}$.
Then locally $\RN{2}_{Y}$ is the Jacobian of $f$,
$$ \left( \frac{\partial^{2} f}{\partial x^{i} \partial x^{j}} \right)_{i,j}.$$
Hence in the notation of the proof of Lemma~\ref{Lemma32} we have 
\[ \mathrm{rank}\, \RN{2}_{Y,y} = \mathrm{rank}\, (f'_{i,j})_{i,j \geq 2}. \]
A point $y \in Y$ is called $r$-\emph{Eckardt} if $\mathrm{rank}\, \RN{2}_{Y,y} \leq r$. If $r=0$, then $y$ is simply called \emph{Eckardt}.
The name originates from Eckardt points on cubic surfaces, which are by definition the points of intersections of its lines.

We see that $\CM_Y$ is the locus of $r$-Eckardt points where $r \leq 2$.
At a $0$-Eckardt point we have
$T_{y}Y\subset R_{y}Y$. These give a $3$-dimensional family of elliptic curves in class $\beta$.
The $1$-Eckardt points give ramification points of the map $\Sigma \to \CM_Y$ and correspond to double planes in $T_{y} Y \cap R_{y} Y$. 
Finally the $2$-Eckardt points correspond to the pairs of elliptic curves $E \cup E'$.
Lemma~\ref{Lemma32} then says that a general cubic fourfolds has no $0$ or $1$-Eckardt points.

A cubic fourfold $Y$ containing an Eckardt point also contains a plane, hence, in Hassett's notation \cite{H}, we have $Y\in \CC_{8}$.
If the cubic fourfold contains 1-Eckardt points or a locus of 2-Eckardt points of dimension larger than $1$, we expect it to be special too.  \qed
\end{rmk}

\subsection{Parametrizing cubic cones} \label{Subsection_Parametrizing_Cubic_cones}
Let $Y$ be a general cubic fourfold. We know that every curve $C \subset F$ of class $\beta$ and of arithmetic genus $1$ corresponds
to a surface $S_C$ which is an intersection $Z = \p^3 \cap Y$ such that $Z$ is a cone.
We describe a $20$-dimensional homogeneous space that parametrizes (birationally)
all pairs $(P,Z)$ of a $3$-plane $\p^3 \cong P \subset \p^5$ together with a cone cubic surface  $Z \subset P$.
We express the curve $\Sigma$ (of cones contained in $Y$) as the zero locus of a regular section of a rank $19$ vector bundle on this homogeneous space.
This will give the proof that $\Sigma$ is non-singular and via adjunction yields a formula for its genus.
Finally we obtain the degree over $\p^1$ by intersecting $\Sigma$ with the divisor of cones over singular plane cubics.

Let $G = \mathrm{Gr}(4,6)$ be the Grassmannian of $3$-planes $\p^3 \subset \p^5$. Let
\[ 0 \to \CK \to \CO_G \otimes \BC^6 \to Q \to 0 \]
be the universal sequence with $\CK$ the universal rank $4$ subbundle.
A point on the projective bundle
\[ \p( \CK) = \Proj \Sym^{\bullet} \CK^{\vee},\quad p : \p(\CK) \to G \]
corresponds to a $3$-plane $\p^3 \subset \p^5$ together with a point $v \in \p^3$.

Given a vector space $V$ and a line $\ell \subset V$, the degree $d$ hypersurfaces in $\p(V)$ which are cones with vertex $[\ell] \in \p(V)$ are canonically parametrized by the
following subspace of degree $d$ polynomials:
\[ \Sym^d (V/\ell)^{\ast}  \subset \Sym^d V^{\ast}. \]
Hence consider the universal sequence on $\p(\CK)$,
\begin{equation} 0 \to \CO_{\p(\CK)}(-1) \to p^{\ast} \CK \to \tilde{Q} \to 0 \label{fsdgsdg} \end{equation}
and let
\[ \CE = \Sym^3( \tilde{Q}^{\ast} ). \]
The projective bundle
\[ q : \p(\CE) \to \p(\CK) \]
parametrizes triples $(Z,v,P)$ where $P \subset \p^5$ is the $3$-plane and $Z \subset P$ is a cone with vertex $v$.
The map from $\p(\CE)$ to the space of cubic cones in $\p^5$ is an isomorphism away
from the locus of cubic cones with more than one vertex (the $3$-plane is always uniquely determined).
Since these are given by cones over the union of three lines (or more degenerate configurations) the curve $\Sigma$ never intersects this locus.
Hence $\Sigma \subset \p(\CE)$.

We write now $\Sigma$ as the zero locus of a section on $\p(\CE)$.
Consider the universal subbundle
\begin{equation} \CO_{\p(\CE)}(-1) \hookrightarrow q^{\ast} \CE \hookrightarrow (p \circ q)^{\ast} \Sym^3 \CK^{\ast} \label{abc} \end{equation}
that over a point describes the ($1$-dimensional span of the) equation cutting out the cone in the $\p^3$.
Here the second inclusion is the natural one obtained from \eqref{fsdgsdg}.
Let $\CF$ be the cokernel of the composition \eqref{abc}.

The section $f \in H^0(\p^5, \CO_{\p^5}(3))$ defining the cubic $Y$ defines via the correspondence of the Grassmannian a section
\[ s_f \in H^0( G, \Sym^3 \CK^{\ast} ) \]
with fiber $f|_{\p^3} \in H^0(\p^3, \CO_{\p^3}(3))$ over the moduli point $[\p^3 \subset \p^5] \in G$.
We pullback $s_f$ to $\p(\CE)$. Then the composition
\[ \hat{s}_f: \CO_{\p(\CE)} \to (p \circ q)^{\ast} \Sym^3 \CK^{\ast} \to \CF \]
vanishes at a point $(\p^3, v, Z)$ if and only if $V(f) \cap \p^3$ is the cone $Z$.
Hence
\[ \Sigma = V(\hat{s}_f). \]

We conclude that $\Sigma$ is the zero locus of a section of a rank $19$ vector bundle on a $20$-dimensional space.
Since by Corollary~\ref{Cor_abc} we know that $\Sigma$ is $1$-dimensional, the section is regular.
Hence by adjunction the canonical sheaf of $\Sigma$ is a line bundle, and thus $\Sigma$ is Cohen-Macaulay (even Gorenstein).
Applying Corollary~\ref{Cor_abc} again we see that $\Sigma$ is smooth.
We now consider the remaining parts of Theorem~\ref{MainThm}.

\begin{lemma} \label{Lemma_genus} $g(\Sigma) = 631$. \end{lemma}
\begin{proof}
By the adjunction formula we have
\[ \omega_{\Sigma} = \Big( \omega_{\p(\CE)} \otimes \det(\CF) \Big) |_{\Sigma} \]
and hence taking degree
\[ 2 g(\Sigma) - 2 = \int_{\p(\CE)} c_{19}(\CF) \cup (K_{\p(\CE)} + c_1(\CF)). \]
This calculation can be performed using the SAGE package 'Chow' \cite{Chow} using the following code:

\scriptsize{
\begin{verbatim}
G=Grass(4,6)
K=G.sheaves["universal_sub"]
PK=ProjBundle(K.dual(), 'y', name='PK')
E=PK.sheaves['universal_sub'].symm(3)
PE = ProjBundle(E.dual(), 'z', name='PE')
C1 = K.dual().symm(3).chern_character()
C2 = PE.sheaves['universal_quotient'].dual().chern_character()
F=Sheaf(PE, ch=C1-C2)
(F.chern_classes()[19] * (PE.canonical_class() + F.chern_classes()[1])).integral()
\end{verbatim} 
}

\end{proof}

We finally compute the degree over the coarse space of $\Mbar_{1,1}$ by counting
how many of the points in $\Sigma$ correspond to cones over singular plane cubics.
For this we need the following lemma:

Let $X$ be a smooth variety, let $V$ be a rank $3$ vector bundle and let $W = \Sym^3( V^{\ast} )$.
The points of $\p(W)$ lying over $x \in X$ canonically parametrize the cubic curves in $\p(V_x) \cong \p^2$.
\begin{lemma} \label{rhsdfg} The divisor $D \subset \p(W)$ parametrizing singular cubic curves has class $12 z - 12 c_1(V)$
where $z = c_1( \CO_{\p(W)}(1) )$.
\end{lemma}
\begin{proof}
We sketch the proof for a lack of reference. 
Let $P = \p(W) \times_X \p(V)$ and let $p_1, p_2$ be the projection to the first and second factor and $q : P \to X$ the projection to $X$.
Taking the derivative gives the $\CO_X$-derivation
\[ d : \Sym^3 V^{\ast} \to (\Sym^2 V^{\ast}) \otimes V^{\ast}. \]
Consider the morphism obtained by pulling back to $P$ and precomposing with the universal subbundle and post-composing with the canonical quotient bundle,
\[
\psi : p_1^{\ast} \CO_{\p(W)}(-1) \to q^{\ast} \Sym^3 V^{\ast} \xrightarrow{q^{\ast} d} q^{\ast} \Sym^2(V^{\ast}) \otimes V^{\ast} \to p_2^{\ast}  \CO_{\p(V)}(2) \otimes q^{\ast} V^{\ast} \]
By the Jacobi criterion, $\psi$ vanishes precisely on the locus of pairs $(C,y)$ where $y$ is a singular point of the cubic $C$.
The class of the vanishing locus of $\psi$ is the Euler class of the bundle
\[ \CO_{P}(1,2) \otimes V^{\ast}. \]
Pushing forward the Euler class by $p_1$ yields the claim.
\end{proof}

\begin{lemma}  \label{lemma_3780}
Let $D \subset \p(\CE)$ be the divisor parametrizing singular cubic curves in $\p(\tilde{Q})$.
Then
\[ D \cdot \Sigma = 3780. \]
\end{lemma} 
\begin{proof} 
This follows from Lemma~\ref{rhsdfg} and the following submission to 'Chow':
\scriptsize{
\begin{verbatim}
c1,c2 = G.gens(); y=PK.gen(); z=PE.gen()
D = 12*z - 12*y + 12*c1
(F.chern_classes()[19]*D).integral()
\end{verbatim}
}
\end{proof}

\begin{lemma}  \label{cfsdf}
Let $F$ be general. Then there are only finitely many singular
curves $C \subset F$ of class $\beta$ and all of them are nodal.
\end{lemma} 
\begin{proof}
Every singular curve in $F$ of class $\beta$ is of arithmetic genus $\geq 1$ and hence
by the discussion in Section~\ref{Subsection_intro_fano_var}
a cone over a plane cubic. In particular, the curve is rational and has either one node or one cusp.
By a dimension count as in Lemma~\ref{Lemma32} the locus
\[ \Omega = \{ (Y, P) \in \p^{55} \times \mathrm{Gr}(4,6) | Y \text{ smooth}, P \cap Y \text{ cone over a singular cubic} \} \]
is of dimension $55$. By Lemma~\ref{lemma_3780} the map $\Omega \to \p^{55}$ is dominant
so the fibers are generically finite. The argument for cuspidal curves is parallel.
\end{proof}

\begin{lemma} \label{deg_lemma}
The degree of $\Sigma \to \p^1$ sending a curve $C$ to its $j$-invariant is of degree $3780$.
\end{lemma}
\begin{proof}
Let $D \subset \p(\CE)$ be the divisor parametrizing singular cubics.
Since $\Sigma$ does not parametrize any cuspidal cubics and is non-singular, the restriction of $D$ to $\Sigma$ is the pullback of
the class of a point from the natural map $\Sigma \to \p^1$, where $\p^1$ is the coarse space of $\Mbar_{1,1}$.
Hence it is enough to compute the intersection pairing of $D$ with $\Sigma$, which we have done in Lemma~\ref{lemma_3780}.
\end{proof}

\begin{proof}[Proof of Theorem~\ref{MainThm}]
As we have discussed above, the curve $\Sigma$ is Cohen-Macaulay and hence smooth by Corollary~\ref{Cor_abc}.
The vertex-assignment $\Sigma \to \CM_Y$ is an \'etale cover of degree $2$. The associated covering involution on $\Sigma$ has the connected quotient $\CM_Y$.
By Lemma~\ref{lemma_3780} all rational curves parametrized by $\Sigma$ are nodal, and finally the genus and the degree to the $j$-line are computed in Lemmata~\ref{Lemma_genus} and~\ref{deg_lemma} respectively.
\end{proof}

\subsection{Moduli of stable maps} \label{Subsection_moduli_of_stable_maps}
Let $F$ be a general Fano of lines and let
\[ f : E \to F \]
be a map from a non-singular smooth elliptic curve in class $\beta$.
The following together with Lemma~\ref{cfsdf} implies Proposition~\ref{Prop_Non_sing}.

\begin{prop}  \label{Prop_normalbundle}Let $N_{E/F} = f^{\ast} T_F / T_E$ be the normal bundle to $f$. Then 
\[ h^0(E, N_{E/F}) = 1. \] 
\end{prop}

As explained in the introduction every map $f : E \to F$ of class $\beta$ from a curve of arithmetic genus $1$
is a closed immersion. Moreover it is given by a family of lines $\ell \subset Y$ through a fixed vertex $v \in Y$.
The idea of the proof is to compare arbitrary deformations of $\ell$ with those which remain incident to $v$.
This is facilitated by the following lemma.

\begin{lemma}
Let $X$ be a smooth variety, let $g : \tilde{X} \to X$ be the blow-up at a point $v \in X$
and let $\iota : D = \p(T_{X,x}) \hookrightarrow \tilde{X}$ be the exceptional divisor.
We have an exact sequence
\[ 0 \to T_{\tilde{X}} \to g^{\ast} T_X \to \iota_{\ast} T_D(-1) \to 0. \]
\end{lemma}
\begin{proof}
Since $X$ is smooth and $g$ is birational we have an exact sequence
\[ 0 \to g^{\ast} \Omega_{X} \to \Omega_{\tilde{X}} \to \Omega_{\tilde{X}/X} \to 0. \]
Dualizing we get
\[ 0 \to T_{\tilde{X}} \to g^{\ast} T_X \to \ext^1_{\CO_X}( \Omega_{\tilde{X}/X}, \CO_{\tilde{X}} ) \to 0. \]
By a direct check $\Omega_{\tilde{X}/X} = \iota_{\ast} \Omega_D$ and by Grothendieck-Verdier duality
\[ \ext^1_{\CO_X}( \iota_{\ast} \Omega_D, \CO_{\tilde{X}} ) = \iota_{\ast} \hom_{\CO_D}( \Omega_D, \omega_D \otimes \omega_{\tilde{X}}^{-1}|_{D} ). \]
Using $\omega_{\tilde{X}} = g^{\ast} \omega_X((n - 1)D)$, $n = \dim X$ yields the claim.
\end{proof}

Let $I_E = \{ (\ell, y) \in E \times Y | y \in \ell\}$ be the incidence correspondence corresponding to $E$.
We write
\[ \pi : I_E \to E \]
for the projection to the Fano side. Via the projection to the second factor
we may view $I_E$ as the blow-up at the vertex of the cone $Z$ of lines parametrized by $E$.
Hence naturally $I_E \subset \tilde{Y}$, where $g : \tilde{Y} \to Y$ is the blow-up of $Y$ at $v$.
We consider the sequence
\[ 0 \to T_{\tilde{Y}} \to g^{\ast} T_Y \to T_{D}(-1) \to 0. \]
Restricting to $I_E$ and quotienting out by $T_{\pi} \subset T_{\tilde{Y}}|_{I_E}$ yields
\begin{equation} 0 \to \tilde{\CM} \to \CM \to T_D(-1) \to 0 \label{hhhh} \end{equation}
where
\[ \tilde{\CM} = \left( T_{\tilde{Y}}|_{I_E} \right) / T_{\pi}, \quad \CM = \left( g^{\ast} T_Y |_{I_E} \right) / T_{\pi}. \]

We want to compute the pushforward by $\pi_{\ast}$ of the sequence \eqref{hhhh}.
By construction $\pi_{\ast} \CM = T_{F|E}$. Since
\[ (R^1 \pi_{\ast} \CM) \otimes k(\ell) = H^1(\ell, N_{\ell/Y} ) = 0 \]
for every $\ell \in E$ we have $R^1 \pi_{\ast} \CM = 0$.
On the other hand,
\[ \pi_{\ast} \tilde{\CM} = T_{F(\tilde{Y})}|_{E} \]
where $F(\tilde{Y})$ is the Fano variety of lines in $\tilde{Y}$.
The Fano $F(\tilde{Y})$ is cut out from the Fano variety $F(\tilde{\p^5}) = \p(T_{\p^5,v}) = \p^4$
by the tangent space to $Y$, the polar quadric and $Y$, i.e.
\[ F(\tilde{Y}) = \p^4 \cap T_{v}(Y) \cap R_v(Y) \cap Y = E \cup E' \]
where $E'$ is the partner of $E$.
Hence we find $\pi_{\ast} \tilde{\CM} = T_E$.
On the other hand
\[ (R^1 \pi_{\ast} \tilde{\CM}) \otimes k(\ell) = H^1(\ell, N_{\ell/Y}(-1) ) =
\begin{cases}
\BC& \text{ if } \ell \text{ is of second type } \\
0 & \text{ otherwise }.
\end{cases}
\]
Since the three points $E \cap E'$ are precisely the lines of second type we get
\[ R^1 \pi_{\ast} \tilde{\CM} = \bigoplus_{\ell \in E \cap E'} \BC_{\ell} \]
where $\BC_{\ell}$ is the skyscraper sheaf at $\ell$. (This may be seen also directly by cohomology and base change:
The fiber $H^0(\ell, N_{\ell/F}(-1)) = T_{E \cup E', \ell}$ is $2$-dimensional precisely at the intersection points $E \cap E'$).
Therefore pushing forward \eqref{hhhh} by $\pi$ yields the exact sequence
\[ 0 \to T_E \xrightarrow{\varphi} T_{F}|E \to T_{D}(-1)|_{E} \to \oplus_{\ell \in E \cap E'} \BC_{\ell} \to 0. \]
The first map $\varphi$ is precisely the differential of $f : E \to F$, so its cokernel is the normal bundle. We hence obtain
\[ 0 \to N_{E/F} \to T_{D}(-1)|_{E} \to \bigoplus_{\ell \in E \cap E'} \BC_{\ell} \to 0. \]

To understand the global sections of this sequence we need the following lemma.
\begin{lemma} \label{global_gen} Let $E \subset \p^3$ be an elliptic curve contained in a $\p^2$. Then
\[ H^0(E, T_{\p^3}(-1)|_E) = \BC^4, \quad H^1( E, T_{\p^3}(-1)|_E ) = \BC. \]
\end{lemma} 
\begin{proof}
Twisting the Euler sequence and restricting to $E$ we get
\[ 0 \to \CO_{E}(-1) \to \CO_E^4 \to T_{\p^3}(-1)|_E \to 0. \]
Taking cohomology the induced map $\alpha: H^1(E, \CO_E(-1)) \to H^1(E, \CO_E^4)$ is Serre dual to
\[ H^0( E, \CO_E^4) \to H^0(E, \CO_E(1)). \]
This sequence is obtained from taking the global section of the restriction of $H^0(\p^3, \CO_{\p^3}(1)) \otimes \CO_{\p^3} \to \CO_{\p^3}(1)$ to $E$.
Since its kernel is precisely the space of hyerplane that contain $E$ which is of dimension $1$, the map is surjective.
We conclude that $\alpha$ is injective which gives the claim.
\end{proof}

It remains to show that
\begin{equation} H^0(E, T_D(-1)|_E) = \BC^4 \to \bigoplus_{\ell \in E \cap E'} H^0(E,\BC_{\ell}) = \BC^3 \label{surj_to_show} \end{equation}
is surjective. For that we need to analyse the map. Let $\ell \in E \cap E'$. The composition of \eqref{surj_to_show} with the projection to the $\ell$-th summand factors as
\[ H^0(E, T_D(-1)|_{E}) \to T_D(-1) \otimes k(\ell) \xrightarrow{\rho} k_{\ell} = H^1(\ell, N_{\ell/Y}(-1)). \]
The map $\rho$ is obtained from the long exact cohomology sequence of
\[ 0 \to N_{\ell/Y}(-v) \to N_{\ell/Y} \to N_{\ell/Y, v} \to 0 \]
where we have written $N_{\ell/Y,v} = N_{\ell/Y} \otimes k(v)$ for the fiber at $v \in \ell$.
Consider the decomposition
\[ N_{\ell/Y} = \CO_{\ell}(1)^2 \oplus \CO_{\ell}(-1). \]
We can hence identify $\rho$ with the projection
\[ N_{\ell/Y,v} \to N_{\ell/Y,v} / (\CO_{\ell}(1)_v \oplus \CO_{\ell}(1)_v) = \CO_{\ell}(-1)_v. \]
The normal directions spanned by the $\CO_{\ell}(1)_v$ summands is the space spanned by the tangent spaces $T_{E,\ell}$ and $T_{E',\ell}$.
We hence need to show that the following map is surjective:
\begin{equation} \label{Dsgdgsg}  H^0(E, T_D(-1)|_{E}) \to \bigoplus_{\ell \in E \cap E'} T_D(-1)_{\ell} / \left( T_{E,\ell} \oplus T_{E',\ell} \right). \end{equation}

To do so we pick coordinates. We can take the equation of the cubic fourfold to be
\[ f = x_0^2 x_1 + x_0 x_1 f_2' + x_0 x_2 x_3 + f_3 \]
where $f_2' \in k[x_1, \ldots, x_5]_1$ and $f_3 \in k[x_1, \ldots, x_5]_3$.
Here $v = [1, 0 ,\ldots, 0]$.
We set $x_0 = 1$ and consider the projectivization of the tangent space at $v$,
\[ D = \p(T_{Y,v}) = \p^3_{x_2, x_3, x_4, x_5}. \]
Inside $D$ we have the complete intersection
\[ E \cap E' = \p( V(x_2, x_3, f_3) ) = \{ \ell_1, \ell_2, \ell_3 \}. \]
By change of coordinates we may assume that the $\ell_i$ are distinct from $[0,0,1,0]$ and hence write
$\ell_i = (0,0,a_i,1)$ for some $a_i \in \BC$.
Let $e_2, \ldots, e_5$ be the basis of
$H^0( T_D(-1)|_{E} )$ corresponding to the basis vectors.
The sequence $\CO_E(-1) \to T_D(-1)|_{E}$ is given by $1 \mapsto x_2 e_2 + \ldots + x_5 e_5$.
Hence at $\ell_i$ its image is $a_i e_4 + e_5$.
Write
\[ T_{E, \ell_i} = \mathrm{Span}( e_2 + \alpha_i e_4), \quad 
T_{E', \ell_i} = \mathrm{Span}( e_3 + \beta_i e_4). \]
Then $i$-th factor of the map \eqref{Dsgdgsg} is given by
\[ e_2 \mapsto -\alpha_i e_4, \quad e_3 \mapsto - \beta_i e_4, \quad e_4 \mapsto e_4, \quad e_5 \mapsto -a_i e_4. \]
Hence \eqref{Dsgdgsg} is represented by the matrix
\begin{equation} \label{t435324} 
\begin{pmatrix}
-\alpha_1 & -\beta_1 & 1 & a_1 \\
-\alpha_2 & -\beta_2 & 1 & a_2 \\
-\alpha_3 & -\beta_3 & 1 & a_3 \\
\end{pmatrix}
\end{equation}
We need to check it is surjective. Set $x_5=1$ and let $g = f_3|_{x_1=0, x_5=1}$.
Then
\begin{align*} \alpha_i & = - \frac{ g_{x_4}(0,0,a_i)}{g_{x_2}(0,0,a_i)} = - \frac{ \prod_{j \neq i} (a_j - a_i)}{g_{x_2}(0,0,a_i)}, \\
\alpha_i & = - \frac{ g_{x_4}(0,0,a_i)}{g_{x_3}(0,0,a_i)} = - \frac{\prod_{j \neq i} (a_j - a_i)}{g_{x_3}(0,0,a_i)}.
\end{align*} 
Since $g_{x_2}(0,0,a_i), g_{x_3}(0,0,a_i), g_{x_4}(0,0,a_i)$ involve the monomials $x_2 x_4^{\bullet}$, $x_3 x_4^{\bullet}$ and $x_4^{\bullet}$
respectively, they can be choosen independently from each other. 
On the other hand, the locus of $4 \times 3$-matrices of rank $\leq 2$ is of codimension $2$ in the corresponding space of matrices.
We conclude that the condition that \eqref{t435324} is not surjective is a codimension $2$ condition on the function $f_3$.

The locus of cubic fourfolds with a point $v \in Y$ with $T_y Y \cap R_y Y$ non-integral and \eqref{t435324} not surjective
is therefore of dimension
\[ \underbrace{40}_{\text{Choice of $f_2, f_3$}} + \underbrace{17}_{\text{Choice of } v, T_v Y, T_v \cap R_{v Y}} + \underbrace{- 2}_{\text{non-surjectivity}} 
+ \underbrace{-1}_{\text{overall scaling}} = 54. \]
Since non-singular cubic fourfolds form an open subset in $\p^{55}$,
this locus can not dominate this open subset. 
The proof of Proposition~\ref{Prop_normalbundle} is complete. 

\section{Generalized Kummer fourfolds} \label{Kummer}
In this section we present an example of a very general polarized hyper-K\"ahler fourfold $X$ with primitive curve class $\beta$ such that
$M_{1,0}(X,\beta)$ is empty.

We begin with some generalities. First, if $[f : C \to X] \in M_{1,0}(X,\beta)$ is a stable map, then the image curve
$f(C)$ is of arithmetic genus $\geq 1$. Hence to show $M_{1,0}(X,\beta)$ is empty it is enough to show that the Hilbert scheme 
$\Hilb^0(X,\beta)$ of $1$-dimensional subschemes $Z \subset X$ satisfying the numerical conditions
\[ [Z] = \beta \in H_2(X,\BZ), \quad \chi(\CO_Z) = 0 \]
is empty.
Second, since $\Hilb^0(X,\beta)$ is projective, if for a deformation $\CX \to B$ the Hilbert scheme of the
very general fiber is non-empty, then the Hilbert scheme of the special fiber is non-empty as well.
Hence it is enough to show that for a special pair $(X,\beta)$ the Hilbert scheme $\Hilb^0(X,\beta)$ is empty.\footnote{The
Hilbert scheme also provides a possible pathway to proving Conjecture~\ref{Conj1} for hyperk\"ahler fourfolds. Indeed, since
in this case the expected dimension of the Hilbert scheme
is $1$, to prove Conjecture~\ref{Conj1} it is enough to show it is precisely of dimension $1$ for one special pair $(X,\beta)$,
for example the Hilbert schemes of 2 points of a K3 surface.}
This is done in the following example.\footnote{This example was pointed out to us by H.~Y.~Lin.}

Let $A$ be a simple principally polarized abelian surface of Picard rank $1$
and let $\mathrm{Km}_2(A)$ the associated generalized Kummer fourfold
defined as the fiber over the origin of the summation map
\[ \Hilb^3(A) \to \mathrm{Sym}^3(A) \xrightarrow{+} A. \]
By the universal property of the Hilbert scheme
every map $g : C \to \mathrm{Km}_2(A)$ from a smooth connected curve $C$
corresponds to a curve $\tilde{C} \subset C \times A$ flat of degree $3$ over $C$.
Since $A$ is of Picard rank $1$ the projection of the class of $\tilde{C}$ is a multiple of the theta divisor.
We let $d = d(C)$ be this multiple and call it the degree of $C$.
We also let
\[ k = k(C) = \chi( \CO_{\tilde{C}} ) - 3 (1-g). \]
where $g$ is the genus of $C$.
We define the class of a curve $C \subset \mathrm{Km}_2(A)$ to be the sum of $m_i \cdot (d(C_i), k(C_i)) \in \BZ^2$ where $C_i$ runs over the normalizations of the (reduced) irreducible components of $C$
and $m_i$ is the multiplicity of that component in $C$.
The class of a curve only depends on the homology class of the curve, see \cite[Lem.2]{HilbK3} for an alternative definition.
We write $\Hilb^0(\mathrm{Km}_2(A), (d,k))$ for the Hilbert scheme of curves of class $(d,k)$ and Euler characteristic $0$.

\begin{prop}
Every (possibly reducible, non-reduced) curve in $\mathrm{Km}_2(A)$ of class $(1,-4)$ is isomorphic to $\p^1$.
In particular, 
\[ \Hilb^0(\mathrm{Km}_2(A), (1,-4)) = \varnothing. \]
\end{prop}
\begin{proof}
Let $C \subset \mathrm{Km}_2(A)$ be a curve in class $(1,-4)$.
Then there exist an irreducible reduced component $C_0 \subset C$ with $d(C_0) = 1$.
Since $d(C) = 1$, the curve $C$ must be reduced at $C_0$. We claim that $C_0$ is isomorphic to $\p^1$ and of class $(1,-4)$, so $C = C_0$.
For this let $B$ be the normalization of $C_0$ and consider the corresponding universal family $\tilde{B} \subset B \times A$.
Since $\tilde{B}$ maps to the abelian surface in degree $1$, its image in $A$ is precisely the genus $2$ curve $\Sigma$ whose Jacobian is $A$.
Hence $\tilde{B} = \Sigma$ and thus $B$ is of genus $\leq 2$. By Riemann-Hurwitz $B$ can not be of genus $2$
and since $A$ is simple it can not be of genus $1$. Hence $B = \p^1$ and the map $\Sigma \to B$ has precisely $8$ branch points so $k(C_0)=-4$.
Finally, the map $\Sigma \to B$ is obtained from the complete linear system of an arbitrary degree $3$ line bundle on $\Sigma$, so none of the fibers are the same. Hence $B \to \mathrm{Km}_2(A)$ is a closed immersion.
\end{proof}

\begin{rmk} With a bit more work one can show that
$\Mbar_{0,0}(\mathrm{Km}_2(A), (1,-4))$ is a disjoint union of copies of the quotients $A/\pm 1$.
\end{rmk}

\end{document}